\documentclass[twoside,12pt]{article}

\usepackage[top=10mm, bottom=10mm, left=15mm, right=15mm]{geometry}

\usepackage[english]{babel}
\usepackage{amsmath,amsthm}
\usepackage{amsfonts}
\usepackage{enumitem}
\usepackage{float}
\usepackage{mathptmx}

\newtheorem{thm}{Theorem}[section]
\newtheorem{cor}[thm]{Corollary}
\newtheorem{lem}[thm]{Lemma}
\newtheorem{prop}[thm]{Proposition}

\theoremstyle{definition}

\theoremstyle{remark}

\numberwithin{equation}{section}
\title{On the equation $ab(ab-1)-na=\Delta^2$}%
\author{Sadegh Nazardonyavi}%
\date{}%
\begin{document}
\maketitle


\begin{abstract}
Let $n$ be a positive integer. We study the diophantine equation $ab(ab-1)-na=\Delta^2$, where $a,b$ are positive integers. We also show that if a system of two congruences is soluble, then an equation which is a translation of Erd\H{o}s-Straus conjecture is soluble.
\end{abstract}
\section{Introduction}
Let $m>1$ be an integer, and suppose $(a,m)=1$. Then $a$ is called a quadratic residue of $m$ if $x^2\equiv a\pmod m$ has a solution, otherwise $a$ is called a quadratic nonresidue of $m$.
If $p$ is an odd prime and $(a,p)=1$, define the Legendre symbol $\left(\frac ap\right)$ (or $(a/p)$) to be $1$ if $a$ is a quadratic residue of $p$ and $-1$ if $a$ is a quadratic nonresidue of $p$. The Jacobi symbol is defined for all odd numbers; if $m=\prod p_i^{k_i}$ and $(a,m)=1$ then $\left(\frac am\right)=\prod\left(\frac a{p_i}\right)^{k_i}$, where $\left(\frac a{p_i}\right)$ are Legendre symbols. Note that if $m$ is not prime, then $\left(\frac am\right)=1$ does not necessarily imply that $a$ is a quadratic residue of $m$; see   \cite[ch. 5]{Adler-1995}, however the Legendre and Jacobi symbols share many properties, for instance:
\begin{thm}\cite{Adler-1995}\label{prop: legendre-jacobi}
Let $m$ and $n$ be odd positive integers.
\begin{enumerate}
  \item If $(a,m) = 1$ and $a\equiv b \pmod m$, then $(a/m) = (b/m)$.
  \item If $(a,m) = (b,m) = 1$, then $(ab/m) = (a/m) (b/m)$. In particular, $(a^2/m)=1$.
  \item $(-1/m)=1$ if and only if $m\equiv1\pmod4$.
  \item $(2/m)=1$ if and only if $m\equiv\pm1\pmod8$.
\end{enumerate}
\end{thm}
Also one of the most important and interesting results in number theory which is also a common property of these two symbols is Law of Quadratic Reciprocity (LQR), cf. \cite[ch. 5]{Adler-1995}. It states that if $(m,n)=1$ are odd positive integers, then
\begin{enumerate}
  \item[] $\left(\frac mn\right) = \left(\frac nm\right)$ if at least one of $m$ and $n$ is of the form $4k+1$, and
  \item[] $\left(\frac mn\right) = -\left(\frac nm\right)$ if both $m$ and $n$ are of the form $4k + 3$.
\end{enumerate}
In this note we study some of the elementary properties of the equation
\begin{equation}\label{ab(ab-1)-na=Delta2}
ab(ab-1)-na=\Delta^2,\qquad 1\leq b<n,
\end{equation}
where $n$ is a given positive integer. If $\Delta=0$, then it is easily seen that the equation is soluble for all $n>1$. Therefore throughout the paper we assume $\Delta>0$ unless otherwise indicated.
\begin{prop}\label{prop: ab(ab-1)-na=Delta2}
Let $n>1$ be a positive integer. Then there exist positive integers $a,b,\Delta$ such that  $ab(ab-1)-na=\Delta^2$
is soluble.
\end{prop}
\begin{proof}
If we let $a=1$, $b=n+1$ and $\Delta=n$, it is done. Now we restrict $1\leq b<n$ as in (\ref{ab(ab-1)-na=Delta2}) and suppose $n>3$ is given. If $n=2k$ then put $a=(k+1)^2$, $b=1$ and $\Delta=k(k+1)$. If $n=4k+1$ then let $a=(k+1)^2$, $b=2$ and $\Delta=(k+1)(2k+1)$. If $n=4k+3$ then let $a=(k+2)^2$, $b=1$ and $\Delta=k(k+2)$.
%
%
%
\end{proof}
When $n$ is a prime, in particular if $n\equiv1\pmod4$, then solutions of (\ref{ab(ab-1)-na=Delta2}) have interesting properties. Proposition \ref{prop: ab(ab-1)-na=Delta2} shows that the equation is soluble, however if we impose some extra condition on variable $b$, i.e. assuming $b\equiv3\pmod4$, then it will be equivalent to (the difficult part of) Erd\H{o}s-Straus conjecture, where obviously it is not proved here.
Here we study some properties of the solutions of (\ref{ab(ab-1)-na=Delta2}).
\begin{prop}\label{prop: ab(ab-1)-pa=Delta2,a nonzero mod p}
Let $p$ be a prime number and there are positive integers $a,b,\Delta$ such that
\begin{equation}\label{ab(ab-1)-pa=Delta2}
ab(ab-1)-pa=\Delta^2.
\end{equation}
If $p|a$, then $p=2$.
\end{prop}
\begin{proof}
Suppose that $p|a$, then $a=pa'$. Hence $\Delta=p\Delta'$. From this we get $a'=pa''$ and
$(pa''b)^2-a''(p+b)=\Delta'^2$. Continuing this procedure we get $a=p^t$ for some $t>0$.
Also we deduce that if $\Delta^2=p^{2k}\Delta'^2$, where $k>0$ and $(p,\Delta')=1$, then $t=2k$. Therefore,
$$
p+b=(p^kb-\Delta')(p^kb+\Delta').
$$
If $k>1$, then we arrive at $b\leq p/(p^k-1)<1$ which is contradiction. Let $k=1$. Then $pb+\Delta'\leq p+b$. From here we get
$$
p\leq\frac{b}{b-1}\leq2.
$$
\end{proof}
\begin{prop}
Let $p$ be a prime number and there are positive integers $a,b$ and $\Delta\geq0$ such that
\begin{equation}\label{ab(ab-1)-pa=Delta2.}
ab(ab-1)-pa=\Delta^2.
\end{equation}
If $p|b$, then $p=2$.
\end{prop}
\begin{proof}
Let $b=pb'$, then by (\ref{ab(ab-1)-pa=Delta2.}) we have $\Delta=p\Delta'$. Also from (\ref{ab(ab-1)-pa=Delta2.}) we deduce $\Delta'\leq ab'-1$. Hence we have
$$
a(b'+1)=p(ab'-\Delta')(ab'+\Delta')\geq p(ab'+\Delta').
$$
So
$$
b'(p-1)+p\Delta'/a\leq1
$$
Since $p$ is prime, we must have $b'=1$, $p=2$ and $\Delta'=0$, and replacing in the equation (\ref{ab(ab-1)-pa=Delta2.}) we get $a=1$.
\end{proof}
The proof of the following result is almost similar to that of the proposition above.
\begin{prop}\label{prop: ab(ab-1)-pa=Delta2,Delta nonzero mod p}
Let $p>1$ be a prime number and there are positive integers $a,b,\Delta$ such that
\begin{equation}\label{ab(ab-1)-pa=Delta2}
ab(ab-1)-pa=\Delta^2,\qquad ,\ 1\leq b<p.
\end{equation}
If $p|\Delta$, then $p=2$.
\end{prop}
\begin{prop}
Let $p$ be a prime number and there are positive integers $a,b,\Delta$, such that
$$
ab(ab-1)-pa=\Delta^2,\qquad p|ab-1.
$$
  If $1\leq a\leq p$, $b\geq1$ then (the only solution is) $a=1$ and $b=p+1$.
\end{prop}
\begin{proof}
Let $ab-1=pc$. Since $p$ is prime, from $ab(ab-1)-pa=\Delta^2$ we have $\Delta=p\Delta'$. Also from $ab(ab-1)-pa=\Delta^2$ we get that $\Delta\leq ab-1$, and so $\Delta'\leq c$. Now from
\begin{equation}\label{c-a=p(Delta'-c)(Delta+c)}
c-a=p(\Delta'-c)(\Delta+c)
\end{equation}
we have $c\leq a$. If $c=a$ we deduce $a=1$, $b=p+1$ and $\Delta=p$. If $c<a$, then by (\ref{c-a=p(Delta'-c)(Delta+c)}) and since $\Delta',\ c$ are positive, we have $c-a\leq-3p$ that leads to $b\leq -2p$ which is contradiction to the assumption $b\geq1$.
\end{proof}
Let $p>2$ be a prime number. By propositions \ref{prop: ab(ab-1)-pa=Delta2,a nonzero mod p}, \ref{prop: ab(ab-1)-pa=Delta2,Delta nonzero mod p} and Theorem \ref{prop: legendre-jacobi}
$$
\left(\frac{ab}p\right)\left(\frac{ab-1}p\right)=\left(\frac{\Delta^2}p\right)=1
$$
Therefore
$$
\left(\frac{ab}p\right)=\left(\frac{ab-1}p\right)=-1\qquad\text{or}\qquad \left(\frac{ab}p\right)=\left(\frac{ab-1}p\right)=1
$$
Before we continue, we give a lemma which is useful in studying some of the properties for the solutions of (\ref{ab(ab-1)-pa=Delta2}).
Consider the following equation:
\begin{equation}\label{(4x-p)yz=x(y+z), erdos-straus equivalent}
(4x-p)yz=x(y+z),
\end{equation}
where $p$ is a given prime and $x,y,z$ are positive integers \footnote{Note that it is one form of the Erd\H{o}s-Straus equation
$$
\frac4p=\frac1x+\frac1{py}+\frac1{pz}.
$$}.
Since $4x-p=x(1/y+1/z)\leq2x$ we have
\begin{equation}\label{x<p/2}
x\leq p/2.
\end{equation}
\begin{lem}\label{lem: (4x-p)yz=x(y+z)}
Let $p\geq2$ be a prime and there are positive integers $x,y,z$ such that
\begin{equation}\label{(4x-p)yz=x(y+z)}
(4x-p)yz=x(y+z).
\end{equation}
Then $yz\not\equiv0\pmod p$.
\end{lem}
\begin{proof}
Suppose that $z\geq y$. If for instance $z=pz'$ then we must have $y=py'$.

\begin{equation}\label{(4x-p)py'z'=x(y'+z')<p/2(y'+z')}
(4x-p)py'z'=x(y'+z')\leq p/2(y'+z').
\end{equation}
If $p=2$, then by (\ref{x<p/2}) the only choice for $x$ is being $x=1$. Then we have $y'=z'=0$, which is impossible since $0=py'=y\geq1$. Now assume $p>2$. The minimum of the left-hand side of (\ref{(4x-p)py'z'=x(y'+z')<p/2(y'+z')}) occurs if $4x-p=1$. Then $y'z'<2/3(y'+z')$ in which holds only if $y'=z'=1$. Therefore $(4x-p)p=2x$ which is impossible since the left-hand side is odd.
\end{proof}
Let $n$ be a positive integer. By Theorem \ref{prop: legendre-jacobi} and LQR one may deduce
\begin{equation}\label{(n/2n+-1)=+-1}
\left(\frac n{4n\pm1}\right)=1.
\end{equation}
\begin{thm}\label{prop (yz/p)=(y+z/p)=+-1}
Let $p>2$ be a prime number and $x,y,z$ be positive integers such that
$$
(4x-p)yz=x(y+z).
$$
If $p\equiv1\pmod4$, then $\left(\frac{yz}p\right)=\left(\frac{y+z}p\right)=-1$, and if $p\equiv3\pmod4$, then $\left(\frac{yz}p\right)=\left(\frac{y+z}p\right)=1$.
\end{thm}
\begin{proof}
Since $(x/p)\neq0$ by (\ref{x<p/2}), using Proposition \ref{prop: legendre-jacobi} one can see that $\left(\frac{yz}p\right)=\left(\frac{y+z}p\right)$. Also by Lemma \ref{lem: (4x-p)yz=x(y+z)} it is seen that $\left(\frac{yz}p\right)\neq0$. Suppose $(y,z)=d$. Then $y=dy'$ and $z=dz'$. It is easily seen that since $(y',z')=1$ then $(y'z',y'+z')=1$.
The equation becomes
$
(4x-p)y'z'd=x(y'+z')
$
hence $y'z'|x$, there exists $x'$ such that $x=x'y'z'$. Then
$
pd+x'y'=x'z'(4y'd-1).
$
Since $p$ is prime from the above equation we have $d=x'd'$. Therefore
\begin{equation}\label{pd'+y'=z'(4x'y'd'-1)}
pd'+y'=z'(4x'y'd'-1)
\end{equation}
and
$
y'\equiv z'(4x'y'd'-1)\pmod p
$.
Thus
\begin{equation}\label{y'z'/p=(4x'y'd'-1)/p}
\left(\frac{y'z'}{p}\right)=\left(\frac{4x'y'd'-1}{p}\right)
\end{equation}
From (\ref{pd'+y'=z'(4x'y'd'-1)}) we have $pd'\equiv-y'\mod(4x'y'd'-1)$. Thus
\begin{equation}\label{pd'/(4x'y'd'-1)=-y'/4x'y'd'-1}
\left(\frac{pd'}{4x'y'd'-1}\right)=\left(\frac{-y'}{4x'y'd'-1}\right)
\end{equation}
By Lemma \ref{lem: (4x-p)yz=x(y+z)} and (\ref{(n/2n+-1)=+-1}), we have
$
\left(\frac{p}{4x'y'd'-1}\right)=-1
$.
By (\ref{y'z'/p=(4x'y'd'-1)/p}) and LQR we then deduce
$$
\left(\frac{yz}{p}\right)=\left(\frac{y'z'd^2}{p}\right)=\left(\frac{y'z'}{p}\right)=\left(\frac{4x'y'd'-1}{p}\right)=\left(\frac{p}{4x'y'd'-1}\right)=-1.
$$
By a similar argument one can demonstrate the respective result for $p\equiv3\pmod4$.
\end{proof}
\begin{prop}\label{prop: ab(ab-1)-pa=Delta2,p=1,b=3(mod4)}
Let $p\equiv 1(\mod4)$ be a prime number. If there are positive integers $a,b,\Delta$ such that
$$
ab(ab-1)-pa=\Delta^2,\qquad ,\ b\equiv 3(\mod4)
$$
then
$$
\left(\frac{ab}{p}\right)=\left(\frac{ab-1}{p}\right)=-1.
$$
\end{prop}
\begin{proof}
Write the equation as
$$
(ab)^2-a(p+b)=\Delta^2.
$$
Set $p+b=4x$ and
\begin{equation}\label{ab=y+z}
ab=y+z,\qquad ax=yz.
\end{equation}
Hence we have
$$
(y+z)^2-4yz=(z-y)^2=\Delta^2.
$$
Multiplying both sides of (\ref{ab=y+z}) we have
$
(4x-p)yz=x(y+z).
$
From Theorem \ref{prop (yz/p)=(y+z/p)=+-1} we deduce
$
\left(\frac{ab}{p}\right)=\left(\frac{y+z}{p}\right)=\left(\frac{yz}{p}\right)=-1
$
\end{proof}
%
In the proof of Proposition \ref{prop: ab(ab-1)-na=Delta2} it was shown that if $p\equiv3\pmod4$ then there are positive integers $a,b,\Delta$ such that
$$
ab(ab-1)-pa=\Delta^2,\qquad ,\ b\equiv 1(\mod4).
$$
Using a similar argument to that of Proposition \ref{prop: ab(ab-1)-pa=Delta2,p=1,b=3(mod4)} one can show that
if $p\equiv3(\mod4)$ is a prime number, and if positive integers $a,b,\Delta$ are such that
$$
ab(ab-1)-pa=\Delta^2,\qquad ,\ b\equiv 1(\mod4)
$$
then
$
\left(\frac{ab}{p}\right)=\left(\frac{ab-1}{p}\right)=1.
$
In Proposition \ref{prop: ab(ab-1)-pa=Delta2,Delta nonzero mod p} we saw that if $p>2$ is prime and equation (\ref{ab(ab-1)-pa=Delta2}) is soluble, then $p\not|\Delta$. Now we consider the case $a|\Delta$. Then we have $\Delta=a\Delta'$.
\begin{cor}
Let $p\equiv 1(\mod4)$ be a prime number. If there are positive integers $a,b,\Delta'$ such that
$$
p=a(b^2-\Delta'^2)-b,\qquad 1<ab<p,\quad a\geq1, \quad b\equiv3(\mod4),\ \Delta'\geq1.
$$
Then
$$
\left(\frac{b^2-\Delta'^2}{p}\right)=-1.
$$
\end{cor}
As we mentioned earlier that we do not prove
$$
ab(ab-1)-pa=\Delta^2,\qquad 1\leq a,\ 1\leq b<p,\ b\equiv3(\mod4),
$$
where $p\equiv1\pmod4$, however we show that with some changes in the coefficients of the equation it has solutions. More precisely we show this observation in the following two propositions.
\begin{prop} For $p\equiv1(\mod4)$ there always exist $a,b,d$ such that the following equation:
$$
ab(ab-1)\equiv \Delta^2(\mod pa),\qquad 1\leq a,\ 1\leq b<p,\ b\equiv3(\mod4)
$$
has solution.
\end{prop}
\begin{proof}
If $p=8k+1$, then put $b=3$, $a=5k+1$ and $\Delta=a$. If $p=8k+5$, then put $b=3$, $a=k+1$ and $\Delta=a$.
\end{proof}

\begin{prop}
For $p\equiv1(\mod4)$ there always exist $a,b,d$ such that the following equation:
$$
(2ab-1)^2\equiv 4 d^2(\mod 4pa+1),\qquad 1\leq a,\ 1\leq b<p,\ b\equiv3(\mod4)
$$
has solution.
%
\end{prop}
\begin{proof}
Let $p=4k+1$. Then set $a=12k+1$, $b=4k-1$ and $\Delta=30k+4$.
\end{proof}
%
%
Let $n>1$ be a positive integer and there exist positive integers $a,c,d,\Delta$ and $1\leq b<n$ such that
$$
\left\{
  \begin{array}{ll}
    ab(ab-1)-nac=\Delta^2,\\[15pt]
    (2ab-1)^2-(4na+1)d=4\Delta^2.
  \end{array}
\right.
$$
Then we will have $(4na+1)d-(4na)c=1$. Hence $c=1+(4na+1)t$ and $d=1+(4na)t$ for some $t\geq0$. Replacing $c$ in the first equation we get
$$
1+(4na+1)t=c=\frac{ab(ab-1)-\Delta^2}{na}<\frac{na(na-1)}{na}=na-1
$$
which yields that $t=0$, and hence $c=d=1$. Therefore we have in particular the following result.
\begin{prop}
Let $n\equiv1(\mod4)$ be a positive integer. If there are positive integers $a,b,\Delta$, where $1\leq b<n$ and $b\equiv3\pmod4$  satisfying the system of equations:
$$
\left\{
  \begin{array}{ll}
    ab(ab-1)-nac=\Delta^2,\qquad c>0 \\[15pt]
    (2ab-1)^2-(4na+1)d=4\Delta^2,\qquad d>0
  \end{array}
\right.
$$
then
$$
ab(ab-1)-na=\Delta^2.\ 
$$
\end{prop}

\section*{Acknowledgment}
To be filled later.
\bibliographystyle{plain}
\bibliography{diophantine-equation-related-to-erdos-straus-conj_ver_01-arxiv}

\vspace*{20pt}
{\textbf{Address:} School of Mathematics, Institute for Research in Fundamental Sciences (IPM), P.O. Box: 19395--5746,
Tehran, Iran;\\
Department of Mathematical Sciences, Isfahan University of Technology (IUT), Isfahan, 84156-83111, Iran\\}
\hspace*{10pt}
\textbf{E-mail:} \texttt{sdnazdi@yahoo.com}\\
\hspace*{14pt}\textbf{Phone:} \texttt{+98(31)3793\,2319}\\
\end{document}